\documentclass[12pt]{amsart}

\usepackage{fullpage}
\usepackage{mathtools}
\usepackage{amssymb,amsmath,amsthm,amscd,mathrsfs,graphicx}
\usepackage[dvipsnames]{xcolor}
\usepackage{bm}
\usepackage{hyperref}
\usepackage{dsfont}
\usepackage{enumerate}
\usepackage{float}
\usepackage{latexsym, amsxtra}
\usepackage{pdfpages}
\usepackage{multicol}
\usepackage[normalem]{ulem}
\usepackage{stmaryrd}
\usepackage{tikz}
\usepackage[T1]{fontenc}
\usepackage{verbatim}
\usepackage{xy}
\usepackage{indentfirst}
\usepackage{tikz-cd}
\usepackage{url}

\makeatletter
\def\thm@space@setup{%
  \thm@preskip=2ex \thm@postskip=2ex
}
\makeatother

\hypersetup{
	hidelinks,
	pdftitle={Small-Subgroup Criteria for Liftability of Projective Automorphism Groups of Smooth Hypersurfaces},
	pdfauthor={Baiting Xie and Zhiwei Zheng}
}

\numberwithin{equation}{section}
\theoremstyle{plain}

\newtheorem{thm}{Theorem~}[section] 
\newtheorem{lem}[thm]{Lemma~}

\newtheorem{prop}[thm]{Proposition~}

\newtheorem{cor}[thm]{Corollary~}

\theoremstyle{remark}
\newtheorem{rmk}[thm]{Remark~}

\theoremstyle{definition}
\newtheorem{defn}[thm]{Definition~}

\newcommand{\KK}{\mathbb{K}}
\newcommand{\ZZ}{\mathbb{Z}}

\newcommand{\PP}{\mathbb{P}}

\newcommand\PGL{\mathrm{PGL}}

\newcommand\id{\mathrm{id}}

\newcommand\I{\mathrm{I}}

\newcommand\SL{\mathrm{SL}}

\newcommand\diag{\mathrm{diag}}
\newcommand\GL{\mathrm{GL}}

\newcommand\ord{\mathrm{ord}}

\newcommand\Lin{\mathrm{Lin}}

\title{Small-Subgroup Criteria for Liftability of
Automorphism Groups of Smooth Hypersurfaces}

 \author[B. Xie]{Baiting Xie}
\address{Tsinghua University, China}
\email{xbt23@mails.tsinghua.edu.cn}

 \author[Z. Zheng]{Zhiwei Zheng}
\address{Tsinghua University, China}
\email{zhengzhiwei@mail.tsinghua.edu.cn}
\date{}

\begin{document}
\bibliographystyle{amsalpha}

\begin{abstract}  
In this paper, building on our previous Sylow criteria, we establish
small-subgroup criteria of liftability and $F$-liftability for the linear automorphism group $G$ of smooth hypersurfaces $X$ over algebraically closed field of characteristic zero. When $\dim X = p-2$ for some odd prime $p$, we prove that ($F$-)liftability of any finite subgroup of $G$ can be tested on its $p$-subgroups of order at most $p^2$. When $X$ is a degree $p$ hypersurface of dimension $2p-2$, we prove that ($F$-)liftability of $G$ can be tested on all its $p$-subgroups of order at most $p^3$, which is sharp
for $p\geq5$. When $p=3$, this bound improves to $9$, giving the corresponding criteria for smooth cubic fourfolds.
\end{abstract}

\maketitle

\section{Introduction}
\label{section: introduction}
Throughout the paper, let $\KK$ be an algebraically closed field of characteristic zero. Let $V$ be a $\KK$-vector space of dimension $N$, and let $F\in S^d(V^*)$ be a nonsingular homogeneous form. Write $\Lin(F)\subseteq\PGL(V)$ for the projective transformations preserving $F$ up to a scalar. We study when finite subgroups of $\PGL(V)$ lift isomorphically to $\GL(V)$ and, for subgroups of $\Lin(F)$, when such a lifting can be chosen to preserve $F$ itself.

Both questions can be approached via group cohomology. Let
$\pi\colon\GL(V)\to\PGL(V)$ be the quotient map. For a finite subgroup
$G\subseteq\PGL(V)$, we have exact sequence
\begin{equation*}
	1\longrightarrow\KK^\times\longrightarrow\pi^{-1}(G)
	\longrightarrow G\longrightarrow1.
\end{equation*}
Its class in $H^2(G,\KK^\times)$ is the obstruction to liftability. If
$G\subseteq\Lin(F)$, every element of $G$ has an $F$-preserving
representative because $\KK$ is algebraically closed. Hence there is a second
central extension
\begin{equation*}
	1\longrightarrow\mu_d\longrightarrow
	\{A\in\GL(V)\mid \pi(A)\in G,\ F\circ A=F\}
	\longrightarrow G\longrightarrow1,
\end{equation*}
where $\mu_d\subseteq\KK^\times$ is the group of $d$-th roots of unity. Its
class in $H^2(G,\mu_d)$ vanishes exactly when $G$ is $F$-liftable.

Our previous paper \cite{xie2026sylowcriterialiftabilityautomorphism}  established Sylow criteria for ($F$-)liftability. Precisely, for $\KK$ algebraically closed in characteristic zero and $F$ smooth, we have that a finite subgroup $G<\Lin(F)$ is ($F$-)liftable if and only if for any $p\mid \gcd(|G|,N,d)$, there exists a Sylow $p$ subgroup of $G$ that is ($F$-)liftable. See \cite[Theorem 1.1, Theorem 1.3]{xie2026sylowcriterialiftabilityautomorphism}. 
Its proof uses restriction--corestriction and transfer
in group cohomology; see \cite[Chapter III, \S9]{Brown1982Cohomology} and
\cite[Theorem 7]{Eckmann1953Transfer}. Starting from this Sylow reduction, we continue to study how to reduce ($F$-)liftability of a $p$-group to its subgroups. For simplicity, we introduce the following notions:

\begin{defn}\label{definition: L and FL}
	Let $p$ be a prime and $k$ be a positive integer. We say a finite subgroup $G\subseteq\PGL(V)$ satisfies the property $\mathrm{L}(p,k)$ if  all its $p$-subgroups of order at most $p^k$ are liftable. We say a finite subgroup $G \subseteq \Lin(F)$ satisfies the property $\mathrm{FL}(p,k)$ if all its $p$-subgroups of order at most $p^k$ are $F$-liftable.
\end{defn}

Then the following question naturally arises: for some fixed $p$-group $G$, find the smallest integer $k$ such that $G$ is liftable (resp. $F$-liftable) if and only if it satisfies the property $\mathrm{L}(p,k)$ (resp. $\mathrm{FL}(p,k)$).

Our first main result anwswers this question for the case of prime dimension $N=p$:

\begin{thm}
	\label{main theorem: subgroup criterion of liftability when N=p}
	Let $p$ be a prime and let $V$ be a $p$-dimensional $\KK$-vector space. A finite subgroup $G\subseteq\PGL(V)$ is liftable if and only if all its subgroups isomorphic to $(\ZZ/p\ZZ)^2$ are liftable, or equivalently, $G$ satisfies the property $\mathrm{L}(p,2)$.
\end{thm}

For $F$-liftability, we prove:

\begin{thm}
	\label{main theorem: subgroup criterion for F-liftability when N=p}
	Let $p$ be a prime, let $V$ be a $p$-dimensional $\KK$-vector space, and let $F\in S^d(V^*)$ be a nonsingular homogeneous form of degree $d\ge 1$. A finite subgroup $G\subseteq\Lin(F)$ is $F$-liftable if and only if all its subgroups isomorphic to $(\ZZ/p\ZZ)^2$ are liftable and all its elements of order $p$ are $F$-liftable, or equivalently, $G$ satisfies the property $\mathrm{L}(p,2)$ and $\mathrm{FL}(p,1)$.
\end{thm}

The case $N=2p$ and $d=p$, where $p$ is an odd prime, is more involved. By the order estimates of Section \ref{section: order estimate}, every element of $p$-power order in $\Lin(F)$ has order at most $p^2$, and every such element that is not $F$-liftable has order $p$. We show that liftability of $\Lin(F)$ can be tested on abelian $p$-subgroups of order at most $p^3$, whereas $F$-liftability can be tested on all $p$-subgroups of order at most $p^3$; see Theorems \ref{main theorem: subgroup criterion of liftability when N=2p} and \ref{main theorem: subgroup criterion of F-liftability when N=2p}. These bounds are sharp for $p\geq5$; see Corollary \ref{corollary: example for main theorem}.

The case $p=3$ is especially relevant geometrically. A nonsingular cubic
form on a six-dimensional vector space defines a smooth cubic fourfold, and
the bound $p^3=27$ in both criteria improves to $p^2=9$:

\begin{thm}
	\label{main theorem: cubic fourfold criteria}
	Let $V$ be a six-dimensional $\KK$-vector space, let
	$F\in S^3(V^*)$ be a nonsingular homogeneous form of degree $3$. Then its projective automorphism group $\Lin(F)$ satisfies the
	following criteria.
	\begin{enumerate}
		\item The group $\Lin(F)$ is liftable if and only if all its subgroups isomorphic to $(\ZZ/3\ZZ)^2$ are liftable, or equivalently, $\Lin(F)$ satisfies the property $\mathrm{L}(3,2)$.
		\item The group $\Lin(F)$ is $F$-liftable if and only if all its subgroups isomorphic to $(\ZZ/3\ZZ)^2$ are liftable and all its elements of order $3$ are $F$-liftable, or equivalently, $\Lin(F)$ satisfies the property $\mathrm{L}(3,2)$ and $\mathrm{FL}(3,1)$.
	\end{enumerate}
\end{thm}

\medskip

{\it Organization:} Section \ref{section: pre} fixes the notation and recalls the Sylow criterion from \cite{xie2026sylowcriterialiftabilityautomorphism}. Section \ref{section: order estimate} combines the previously established reduction-to-Klein method with new order estimates for individual elements. Section \ref{section: complex projective linear representation} develops the representation-theoretic decomposition tools used below. The criteria in prime dimension and in the case $(N,d)=(2p,p)$ are proved in Sections \ref{section: N=p} and \ref{section: N=2p and d=p}, respectively, including the improved bound for $p=3$. Section \ref{section: example} proves sharpness for $p\geq5$.

\section{Preliminary Results}
\label{section: pre}
We retain the terminology of \cite[\S2]{xie2026sylowcriterialiftabilityautomorphism}, but recall the minimum needed for this paper. Let $V$ be an $N$-dimensional $\KK$-vector space and let $\pi\colon\GL(V)\to\PGL(V)$ be the quotient map. A subgroup $\widetilde G\subseteq\GL(V)$ is a \emph{lifting} of $G\subseteq\PGL(V)$ if $\pi|_{\widetilde G}$ is an isomorphism onto $G$; the group $G$ is \emph{liftable} when such a subgroup exists. For an element $g$ of finite order, a lifting is equivalently a matrix in $\pi^{-1}(g)$ having the same order as $g$.

For $m\geq1$, write $\mu_m\subseteq\KK^\times$ for the group of $m$-th roots of unity. We denote the identity on a vector space $W$ by $\I_W$ and use the commutator convention $[A,B]=ABA^{-1}B^{-1}$.

For a homogeneous form $F \in S^{d}(V^{*})$, set
\begin{equation*}
	\Lin(F)=\{\pi(\varphi)\mid \varphi\in\GL(V),\ \exists\lambda\in\KK^\times,\ F\circ\varphi = \lambda F\}.
\end{equation*}
A matrix is \emph{$F$-preserving} if $F\circ\varphi=F$. An \emph{$F$-lifting} of $G\subseteq\Lin(F)$ is a lifting consisting of $F$-preserving matrices, and $G$ is \emph{$F$-liftable} if it has one. The same terms for an element refer to its cyclic subgroup. Finally, $F$ is \emph{nonsingular} when its projective zero locus is smooth; equivalently, no nonzero zero of $F$ is a common zero of all its first partial derivatives.

We also record the finiteness fact used whenever the full group $\Lin(F)$ occurs below. If $N\geq3$, $d\geq3$, and $(N,d)\neq(3,3),(4,4)$, then $\Lin(F)$ is finite by results of Matsumura--Monsky \cite[Theorems 1 and 2]{matsumura1963automorphisms} and Chang \cite{chang1978plane}. The only excluded case needed in this paper is $(N,d)=(3,3)$; if $X=\{F=0\}\subseteq\PP(V)$, then $\Lin(F)$ identifies with the automorphism group of the polarized elliptic curve $(X,\mathcal O_X(1))$, which is finite. Thus every use below of a Sylow subgroup of a full group $\Lin(F)$ concerns a finite group.

For reference, we quote the Sylow criterion of our previous paper in precisely the form used below.

\begin{lem}[{\cite[Theorems 1.1 and 1.3]{xie2026sylowcriterialiftabilityautomorphism}}]
	\label{lemma: Sylow criterion of liftability}
	Let $V$ be an $ N $-dimensional $\KK$-vector space, and let $F\in S^{d}(V^{*})$ be a nonsingular homogeneous form of degree $d\geq1$. Let $ G $ be a finite subgroup of $ \PGL(V) $. Then:
	\begin{enumerate}
		\item $G$ is liftable if and only if, for every prime $p$ dividing $ N $, the group $G$ has a liftable Sylow $p$-subgroup.
		\item When $ G \subseteq \Lin(F) $, the group $G$ is $F$-liftable if and only if, for every prime $p$ dividing $\gcd(N,d)$, it has an $F$-liftable Sylow $p$-subgroup.
	\end{enumerate}
\end{lem}

\begin{cor}
	\label{corollary: liftability of p groups when gcd(d,N,p)=1}
		Let $V$ be an $ N $-dimensional $\KK$-vector space, and let $F\in S^{d}(V^{*})$ be a nonsingular homogeneous form of degree $d\geq2$. Let $ G $ be a $ q $-subgroup of $ \PGL(V) $, where $q$ is prime. Then:
	\begin{enumerate}
		\item  If $q \nmid N$, then $G$ is liftable.
		\item When $ G \subseteq \Lin(F) $, if $q \nmid \gcd(N,d)$, then $G$ is $F$-liftable.
	\end{enumerate}
\end{cor}

\section{Order Estimates for Individual Elements}
\label{section: order estimate}

In this section, fix an $ N $-dimensional $\KK$-vector space $ V $, a nonsingular homogeneous form $ F \in S^{d}(V^{*}) $ and a prime $ p $ dividing $ d $. For a nonzero integer $n$, let $\nu_p(n)$ denote its $p$-adic valuation, and put $\nu_p(0)=+\infty$. We study the structure of a single element $ g \in \Lin(F) $ of order $ p^{r} $. Our main goal is Theorem \ref{theorem: order of elements in Lin(F)}.

In general, $ g $ may not be $ F $-liftable. However, the following lemma shows that one can always find an $ F $-preserving preimage of $ g $ whose order is a power of $ p $.
\begin{lem}
	\label{lemma: existence of F-preserving p-preimages for single elements}
	Let $p$ be a prime dividing $d$, and let $g \in \Lin(F)$ be an element of order $p^{r}$, where $r\geq1$. Suppose that $ g $ is not $ F $-liftable. Let $ c $ be the minimal positive integer such that $ g^{p^{c}} $ is $ F $-liftable. Then $ r-\nu_p(d) < c \leq r $. Furthermore, setting $ s = c+\nu_p(d) $, there exists an element $ \widetilde{g} \in \pi^{-1}(g) $ of order $ p^{s} $ such that $ F \circ \widetilde{g} = F $.
\end{lem}

\begin{proof}
	Since $g^{p^r}=1$ is $F$-liftable, we have $c\leq r$. Choose an $F$-lifting $B$ of $g^{p^c}$, and let $C\in\pi^{-1}(g)$. There is a scalar $a\in\KK^\times$ such that $C^{p^c}=aB$. Choose $u\in\KK^\times$ with $u^{p^c}=a^{-1}$ and set $A=uC$. Then $A^{p^c}=B$. Since $B$ has order $p^{r-c}$ and $\pi(A)=g$ has order $p^r$, it follows that $A$ has order $p^r$.

	Write $F \circ A=\zeta F$. Since $A^{p^c}=B$ preserves $F$, the order of $\zeta$ divides $p^c$. If $\zeta$ had order $p^a$ for some $a<c$, then $A^{p^a}$ would be an $F$-preserving matrix of order $p^{r-a}$ lifting $g^{p^a}$, contrary to the minimality of $c$. Thus $\zeta$ is a primitive $p^c$-th root of unity.

	The map $z\mapsto z^d$ is surjective on the group of all roots of unity of $p$-power order. Hence there is such a root $\xi$ with $\xi^d=\zeta^{-1}$, and necessarily $\xi$ has order $p^s$, where $s=c+\nu_p(d)$. Set $\widetilde g=\xi A$. Then $\pi(\widetilde g)=g$ and $F \circ \widetilde{g}=F$. If $s\leq r$, then $\widetilde g^{p^r}=1$, contrary to the fact that $g$ is not $F$-liftable. Hence $s>r$, which implies that $c>r-\nu_p(d)$. Then $\widetilde g^{p^r}=\xi^{p^r}$ has order $p^{s-r}$, so $\ord(\widetilde g)=p^s$. 
\end{proof}

Lemma \ref{lemma: existence of F-preserving p-preimages for single elements} reduces the problem to an $F$-preserving matrix. The reduction-to-Klein method originates in \cite[Theorem 4.2]{zheng2022abelian} and was developed further in our previous paper. We need only the following normal-form consequence, for which we first introduce some notation.

\begin{defn}
	\label{definition: standard form for elements in Lin(Klein)}
	Let $d$ be a positive integer and let $p$ be a prime dividing $d$. For each nonnegative integer $r$, define $\ell_{p,d}(r)$ to be the smallest positive integer such that $p^r\mid (1-d)^{\ell_{p,d}(r)}-1$, or equivalently, the multiplicative order of $1-d$ modulo $p^r$. If $\zeta$ is a root of unity of order $p^r$, define
	\begin{equation*}
		D_{p,d}(\zeta) = \diag(\zeta,\zeta^{1-d},\zeta^{(1-d)^{2}},\cdots,\zeta^{(1-d)^{\ell_{p,d}(r)-1}}).
	\end{equation*}
\end{defn}

\begin{lem}[{\cite[Lemmas 4.3 and 4.4]{xie2026sylowcriterialiftabilityautomorphism}}]
	\label{lemma: normal form of F-preserving elements}
	Let $\varphi\in\GL(V)$ be an element of order $p^{r}$ such that
	$F \circ \varphi=F$. Then there exists a basis of $V$ where $\varphi$ is represented by a diagonal matrix
	\begin{equation*}
		\diag(D_{p,d}(\zeta_{1}),\cdots,D_{p,d}(\zeta_{m})),
	\end{equation*}
	where $ \zeta_{1},\cdots,\zeta_{m} $ are roots of unity of order dividing $p^{r}$.
\end{lem}

When $d=2$, necessarily $p=2$, and $\ell_{2,2}(r)=1$
for $r=0,1$, while $\ell_{2,2}(r)=2$ for $r\geq2$.
When $d\geq3$, the following lemma computes $\ell_{p,d}(r)$.

\begin{lem}
	\label{lemma: computation of the order of 1-d wrt p^r}
	For every nonnegative integer $r$, integer $d\geq 3$ and prime $p$ dividing $d$, the following statements hold.
	\begin{enumerate}
		\item When $ p \geq 3 $ or $ p = 2 $ and $ 4 \mid d $, we have
		\begin{equation*}
			\ell_{p,d}(r) = \begin{cases}
					1, &r \leq \nu_p(d); \\
				p^{r-\nu_p(d)}, &r > \nu_p(d).
			\end{cases}
		\end{equation*}
		\item Otherwise, $ p = 2 $ and $ 4 \nmid d $. Then 
			\begin{equation*}
			\ell_{2,d}(r) = \begin{cases}
				1, & r =0,1; \\
				2, & 2 \leq r \leq \nu_2(d-2); \\
				2^{r-\nu_2(d-2)}, & r > \nu_2(d-2).
			\end{cases}
		\end{equation*}
	\end{enumerate}
\end{lem}

\begin{proof}
	When $ p \geq 3 $ or $ p = 2 $ and $ 4 \mid d $, it is straightforward to check that, for every $ k\geq1 $,
	\[
	\nu_p\bigl(1-(1-d)^{k}\bigr)
	=\nu_p(d)+\nu_p(k),
	\]
	which gives (1).

	Otherwise, $ p = 2 $ and $ 4 \nmid d $. Then $ 4 \mid d-2 $, and it is straightforward to check that
	\[
	\nu_2\bigl(1-(1-d)^{k}\bigr) = \begin{cases}
		1, & 2 \nmid k; \\
		\nu_2(d-2)+\nu_2(k), & 2 \mid k.
	\end{cases}
	\]
	This gives (2).
\end{proof}

Then the normal form in Lemma \ref{lemma: normal form of F-preserving elements} gives bounds on the orders of elements.

\begin{thm}
	\label{theorem: order of elements in Lin(F)}
	Let $F \in S^{d}(V^{*})$ be a nonsingular homogeneous form of degree $d\ge 3$. Let $p$ be a prime dividing $d$, and let $g\in \Lin(F)$ be an element of order $p^{r}$, where $r\geq1$.
	\begin{enumerate}
		\item If $ g $ is $ F $-liftable, then
			\begin{equation*}
			r < \begin{cases}
				\nu_2(d-2)+\log_{2}(N), & p=2 \text{ and } 4 \nmid d; \\
				\nu_p(d)+\log_{p}(N), & \text{otherwise.}
			\end{cases}
		\end{equation*}
	\item If $ g $ is not $ F $-liftable, then 	\begin{equation*}
		r \leq \begin{cases}
			\nu_2(d-2)-1+\nu_2(N), &p=2 \text{ and } 4 \nmid d; \\
			\nu_p(d)-1+\nu_p(N), & \text{otherwise.}
		\end{cases}
	\end{equation*}
	\end{enumerate}
\end{thm}

\begin{proof}
	(1)	By Lemma \ref{lemma: computation of the order of 1-d wrt p^r}, it suffices to prove that $ N > \ell_{p,d}(r) $. Since $ g $ is $ F $-liftable, there exists an $ F $-lifting $ \widetilde{g} $ of $ g $. Applying Lemma \ref{lemma: normal form of F-preserving elements} to
	$\widetilde{g}$, we can choose a basis of $V$ where $\widetilde{g}$ is represented by a diagonal matrix
	\begin{equation*}
		\diag(D_{p,d}(\zeta_{1}),\cdots,D_{p,d}(\zeta_{m})),
	\end{equation*}
	where $ \zeta_{1},\cdots,\zeta_{m} $ are $ p^{r} $-th roots of unity.
		 Since $\ord(\widetilde{g})=p^r$, some
		$\zeta_j$, say $\zeta_1$, is a primitive $p^r$-th root of unity. Hence $ N \geq \ell_{p,d}(r) $. If the equality holds, then $ \widetilde{g} $ is represented by $ D_{p,d}(\zeta_{1}) $. Note that for every integer $k\geq0$,
	\begin{equation*}
		\nu_p(1-(1-d)^{k}) \geq 1,
	\end{equation*}	
	which implies that 
	$ D_{p,d}(\zeta_{1})^{p^{r-1}} $ is a scalar matrix. This contradicts $\ord(g)=p^r > p^{r-1} $. Hence $ N > \ell_{p,d}(r) $. 
	
	(2) By Lemma \ref{lemma: computation of the order of 1-d wrt p^r}, it suffices to prove that $ \ell_{p,d}(r+1) \mid N $. By Lemma
	\ref{lemma: existence of F-preserving p-preimages for single elements}, there exists $ s > r $ and
	$\widetilde{g}\in\pi^{-1}(g)$ of order $p^s$ such that
	$F\circ\widetilde{g}=F$. As in (1), we may assume that $\widetilde{g}$ is represented by a diagonal matrix
	\begin{equation*}
		\diag(D_{p,d}(\zeta_{1}),\cdots,D_{p,d}(\zeta_{m})),
	\end{equation*}
	where $ \zeta_{1},\cdots,\zeta_{m} $ are $ p^{s} $-th roots of unity.

	Since $ \ord(g) = p^{r} $, $\widetilde g^{p^r}$ is a scalar transformation of order $p^{s-r}$. Thus, for each $ j $,  $\zeta_{j}^{p^{r}} $ has
	order $p^{s-r}$, which implies that $ \zeta_{j} $ is a primitive $ p^{s} $-th root of unity. Hence
	$ \ell_{p,d}(s) $ divides $ N $. Since $r+1\leq s$, reduction modulo $p^{r+1}$ shows that $\ell_{p,d}(r+1)$ divides $\ell_{p,d}(s)$, and therefore $ \ell_{p,d}(r+1) \mid N $.
\end{proof}

\section{Projective Representations of $p$-Groups and Liftability}
\label{section: complex projective linear representation}

In this section, let $V$ be an $N$-dimensional $\KK$-vector space, let $F\in S^d(V^*)$ be a nonsingular homogeneous form of degree $d\geq2$, and fix a prime $p$. The additional condition $p\mid d$ will be imposed explicitly where it is needed.

\begin{lem}
	\label{lemma: cover of p-groups}
	Let $ G \subseteq \PGL(V) $ be a $ p $-subgroup. Then there exists a $ p $-subgroup $ \widetilde{G} \subseteq \pi^{-1}(G) $ such that $ \pi(\widetilde{G}) = G $. 
\end{lem}

\begin{proof}
	Because $\KK$ is algebraically closed, every element of $G$ has a determinant-one representative: if $A$ represents it, multiply $A$ by a scalar $u$ satisfying $u^N\det(A)=1$. Hence $H=\pi^{-1}(G)\cap\SL(V)$ is a finite group fitting into an exact sequence
	\begin{equation*}
		1\longrightarrow \mu_N\longrightarrow H\longrightarrow G\longrightarrow1.
	\end{equation*}
	If $\widetilde G$ is a Sylow $p$-subgroup of $H$, then its image in the $p$-group $G$ has index dividing both a power of $p$ and the prime-to-$p$ index $[H:\widetilde G]$. Hence $\pi(\widetilde G)=G$.
\end{proof}

\begin{lem}
	\label{lemma: F-preserving cover for p-groups}
		Let $ G \subseteq \Lin(F) $ be a $ p $-subgroup. Then there exists an $ F $-preserving $ p $-subgroup $ \widetilde{G} \subseteq \pi^{-1}(G) $ such that $ \pi(\widetilde{G}) = G $. 
\end{lem}

\begin{proof}
		Set
	\begin{equation*}
		\overline{G} = \{\varphi \in \pi^{-1}(G) \mid F \circ \varphi = F\}.
	\end{equation*}
	If $A$ represents an element of $G$ and $F\circ A=\lambda F$, choose $u\in\KK^\times$ with $u^d\lambda=1$. Then $uA$ is an $F$-preserving representative. Thus $\overline G$ maps onto $G$, and its kernel is the finite scalar group $\mu_d$. In particular, $\overline G$ is finite. A Sylow $p$-subgroup $\widetilde G$ of $\overline G$ maps onto $G$ by the same index argument as in Lemma \ref{lemma: cover of p-groups}.
\end{proof}

These finite $p$-group covers allow us to apply ordinary representation theory to projective $p$-group actions.

\begin{cor}
	\label{corollary: dimension of irreducible representations of G}
	Let $G\subseteq\PGL(V)$ be a $p$-subgroup. Then the dimension of every irreducible $\pi^{-1}(G)$-subrepresentation of $V$ is a power of $p$.
\end{cor}

\begin{proof}
	Choose a $p$-subgroup $\widetilde G\subseteq\pi^{-1}(G)$ mapping onto $G$ as in Lemma \ref{lemma: cover of p-groups}, and let $W\subseteq V$ be an irreducible $\pi^{-1}(G)$-subrepresentation. Since every element of $\pi^{-1}(G)$ differs from an element of $\widetilde G$ by a scalar, the two groups have the same invariant subspaces in $V$. In particular, $W$ is irreducible as a $\widetilde G$-representation. By Frobenius's divisibility theorem \cite[Chapter 6, Corollary 2 to Proposition 16]{serre1977linear}, the dimension of an irreducible representation of a finite group over $\KK$ divides the order of the group, hence $\dim W\mid |\widetilde G|$.
	Since $\widetilde G$ is a finite $p$-group, $|\widetilde G|$ is a power of $p$, and therefore so is $\dim W$.
\end{proof}

The following lemma relates liftability to invariant subspaces.
\begin{lem}
	\label{lemma: liftability and decompositions}
	Let $G$ be a $ p $-subgroup of $\PGL(V)$. Let $V_{1}$ be a nonzero $\pi^{-1}(G)$-invariant subspace of $V$.  Denote by $ \pi_{1} \colon \GL(V_{1}) \rightarrow \PGL(V_{1}) $ the canonical projection onto $ \PGL(V_{1}) $ and define
	\begin{equation*}
		G_{1} = \{\pi_{1}(\varphi|_{V_{1}}) \mid \varphi \in \pi^{-1}(G)\} \subseteq \PGL(V_{1}).
	\end{equation*}
	Then $ G_{1} $ is a $ p $-subgroup of $ \PGL(V_{1}) $. Furthermore, if $ G_{1} $ is liftable, then $G$ is liftable. 
\end{lem}

\begin{proof}
	The restriction map $ \pi^{-1}(G) \rightarrow \GL(V_{1}) $ induces a group homomorphism $ G \rightarrow \PGL(V_{1}) $, whose  image is exactly $ G_{1} $. As a quotient of $ G $, $ G_{1} $ is a $ p $-group.
	
	If $ G_{1} $ is liftable, choose a lifting $ \widetilde{G_{1}} $ of $ G_{1} $. For each $g\in G$, there is a unique element $\widetilde g\in\pi^{-1}(g)$ whose restriction to $V_1$ lies in $\widetilde{G_1}$: starting with any representative, a unique scalar gives the prescribed lift of its image in $G_1$. We claim that $ \widetilde{G} = \{\widetilde{g} \mid g \in G\} $ is a lifting of $ G $. For $g,h\in G$, there exists $ \zeta \in \KK^{\times} $ such that $ \widetilde{g}\widetilde{h} = \zeta \widetilde{gh} $. Since
	\begin{equation*}
		\zeta\I_{V_{1}} = (\widetilde{g}\cdot\widetilde{h})|_{V_{1}} \cdot (\widetilde{gh}|_{V_{1}})^{-1} \in \widetilde{G_{1}},
	\end{equation*}
	the scalar $\zeta\I_{V_1}$ lies in $\widetilde{G_1}$. A lifting intersects the scalar subgroup only in the identity, so $\zeta=1$. Hence $\widetilde G$ is a lifting of $G$.
\end{proof}

The same argument has an $F$-liftability version.
\begin{lem}
	\label{lemma: F-liftability and decompositions}
		Let $G$ be a $ p $-subgroup of $\Lin(F)$. Let $V_{1}$ be a $\pi^{-1}(G)$-invariant subspace of $V$. Define $ G_{1} $ as in Lemma \ref{lemma: liftability and decompositions}. Denote by $ F_{1} \in S^{d}(V_{1}^{*}) $ the restriction of $ F $ to $ V_{1} $. Suppose that $ F_{1} \neq 0 $. Then $ G_{1} \subseteq \Lin(F_{1}) $. Furthermore, if $G_{1}$ is $F_{1}$-liftable, then $G$ is $F$-liftable.
\end{lem}

\begin{proof}
		For every $\varphi \in \pi^{-1}(G)$, write $ F \circ \varphi = \zeta F $. Then, for every $ v \in V_{1} $, since both $ \varphi(v) $ and $ v $ lie in $ V_{1} $, we have
	\begin{equation*}
		(F_{1} \circ (\varphi|_{V_{1}}))(v) = F_{1}(\varphi(v)) = F(\varphi(v)) = \zeta F(v) = \zeta F_{1}(v).
	\end{equation*} 
This implies that $\pi_{1}(\varphi|_{V_{1}})\in\Lin(F_{1})$, and therefore $G_{1} \subseteq \Lin(F_{1})$.
	
	If $G_{1}$ is $F_{1}$-liftable, pick an $F_{1}$-lifting $\widetilde{G_{1}}$ of $G_{1}$. As in the proof of Lemma \ref{lemma: liftability and decompositions}, the uniquely normalized representatives whose restrictions lie in $\widetilde{G_1}$ form a lifting $\widetilde G$ of $G$. For each $\varphi \in \widetilde{G}$, write $ F \circ \varphi = \zeta F $. Since $ F_{1} \neq 0 $, there exists $ v \in V_{1} $ such that $ F(v) = F_{1}(v) \neq 0 $. Then 
\begin{equation*}
	\zeta = \frac{(F\circ\varphi)(v)}{F(v)}=\frac{(F_{1}\circ(\varphi|_{V_{1}}))(v)}{F_{1}(v)}=1,
\end{equation*}
which implies that $ \widetilde{G} $ is an $ F $-lifting of $ G $.
\end{proof}

The following lemma produces invariant subspaces from a central projective transformation.
\begin{lem}
	\label{lemma: the space decomposition}
		Let $G$ be a subgroup of $\PGL(V)$. Let $ g $ be an element of the center of $ G $. If, for every $ h \in G $, the abelian subgroup $ \langle g,h \rangle $ generated by $ g $ and $ h $ is liftable, then, for every $\widetilde{g} \in \pi^{-1}(g)$, all eigenspaces of $ \widetilde{g}$ are $\pi^{-1}(G)$-invariant subspaces.
\end{lem}

\begin{proof}
	Fix $\varphi\in\pi^{-1}(G)$ and set $L=\langle g,\pi(\varphi)\rangle$. Since $L$ is abelian and liftable, any lifting $\widetilde L$ of $L$ is abelian and $\pi^{-1}(L)=\KK^\times\widetilde L$ is also abelian. Hence $[\widetilde g,\varphi]=\I_V$. Every eigenspace of $\widetilde g$ is therefore $\varphi$-invariant, and the conclusion follows as $\varphi$ was arbitrary.
\end{proof}

\begin{lem}
	\label{lemma: decomposition of F}
	Let $F\in S^d(V^*)$ be nonsingular, and let $ p $ be a prime dividing $ d $. Let $ \varphi \in \GL(V) $ be an element of order $ p $ such that $ F \circ \varphi = F $. Let $ V_{1} $ be a nonzero eigenspace of $ \varphi $. Then the restriction $ F_{1} $ of $ F $ to $ V_{1} $ is a nonsingular homogeneous form. In particular, it is not zero.
\end{lem}

\begin{proof}
	Choose a basis $ e_{0},\cdots,e_{N-1} $ of $ V $ such that $ \varphi $ is represented by a diagonal matrix. Without loss of generality we may assume $ V_{1} $ is spanned by $ e_{0},\cdots,e_{k-1} $, where $ k \geq 1 $. Denote by $ x_{0},\cdots,x_{N-1} $ the dual basis of $ e_{i} $ and $ f $ the representative polynomial of $ F $ under this basis. Let $\alpha$ be the eigenvalue of $\varphi$ on $V_1$. Since $p\mid d$, we have $\alpha^d=1$. A degree-$d$ monomial of degree $d-1$ in the $V_1$-variables and degree $1$ in a variable from an eigenspace with eigenvalue $\beta\neq\alpha$ has weight $\alpha^{d-1}\beta=\alpha^{-1}\beta\neq1$. It therefore cannot occur in the $\varphi$-invariant form $F$. Consequently, there exists $ f_{1} \in \KK[x_{0},\cdots,x_{k-1}] $ such that
	\begin{equation*}
		f(x_{0},\cdots,x_{N-1}) \equiv f_{1}(x_{0},\cdots,x_{k-1}) \pmod{(x_{k},\cdots,x_{N-1})^{2}}.
	\end{equation*}
	
	If a nonzero point of $V_1$ were singular for $f_1=F|_{V_1}$, the congruence would make all tangential and normal first derivatives of $F$ vanish there, contradicting the nonsingularity of $F$. Thus $ f_{1} $ is nonsingular and represents $ F_{1} = F|_{V_{1}} $ under the basis $ e_{0},\cdots,e_{k-1} $.
\end{proof}

\section{Liftability in Prime Dimension}
\label{section: N=p}
In this section, fix a prime $ p $, a $ p $-dimensional $\KK$-vector space $ V $, and a nonsingular homogeneous form $F\in S^{d}(V^{*})$ of degree $d \geq 1$.

\begin{prop}
	\label{proposition: liftability of p-groups when N=p}
	Let $G$ be a $p$-subgroup of $\PGL(V)$. The following statements are equivalent:
	\begin{enumerate}
		\item $ G $ is liftable.
		\item  all abelian subgroups of $G$ are liftable.
		\item $ \pi^{-1}(G) $ is abelian.
	\end{enumerate}
\end{prop}

\begin{proof}
	Every subgroup of a liftable group is liftable, so (1) implies (2).
	
	Assume (2). If $G$ is trivial, then (3) is immediate. Otherwise, choose an element $g$ of order $p$ in the center of $G$. For every $h\in G$, the subgroup $\langle g,h\rangle$ is abelian and hence liftable by (2). Lemma \ref{lemma: the space decomposition} therefore shows that the eigenspaces of a lifting of $g$ are $\pi^{-1}(G)$-invariant. Since $g$ is nontrivial, this gives a nontrivial decomposition of $V$ as a $\pi^{-1}(G)$-representation. If $V=V_1\oplus\cdots\oplus V_k$ is an irreducible decomposition, Corollary \ref{corollary: dimension of irreducible representations of G} implies that every $\dim V_i$ is a power of $p$. Since $\dim V=p$ and $k>1$, all $V_i$ are one-dimensional. Hence $\pi^{-1}(G)$ is abelian, proving (3).
	
	We finally prove that (3) implies (1). Since $\pi^{-1}(G) $ is abelian, the group $G$ is abelian. By the structure theorem of finite abelian groups, there exist $ g_{1},\cdots,g_{k} \in G $ such that
	\begin{equation*}
		G = \langle g_{1} \rangle \times \cdots \times \langle g_{k} \rangle.
	\end{equation*}
	
	For each $i$, choose a preimage $\widetilde g_i$ having the same order as $g_i$. These preimages commute because $\pi^{-1}(G)$ is abelian. Then the group
	\begin{equation*}
		\widetilde{G} = \langle \widetilde{g_{1}},\ldots,\widetilde{g_{k}} \rangle
	\end{equation*}
	maps isomorphically onto $G$: a scalar element of $\widetilde G$ would give a relation among the independent cyclic factors of $G$, and the order-preserving choice of the $\widetilde g_i$ then makes that element the identity. Thus $ G $ is liftable.
\end{proof}

\begin{lem}
	\label{lemma: structure of semi-commutative pair when N=p}
	Suppose that $\varphi,\psi\in\GL(V)$ and $\zeta\neq1$ satisfy $\varphi\psi=\zeta\psi\varphi$. Then $\zeta$ is a primitive $p$-th root of unity, and there exist a basis $e_0,\ldots,e_{p-1}$ of $V$ and $c_1,c_2\in\KK^\times$ such that
	\begin{equation*}
		\varphi(e_{i}) = c_{1}\zeta^{i}e_{i},\quad \psi(e_{i}) = c_{2}e_{i+1}, \quad e_{p}=e_{0}.
	\end{equation*}
In particular, $ \varphi^{p} $ and $ \psi^{p} $ are scalar transformations. 
\end{lem}

\begin{proof}
	The equality $ \varphi\psi = \zeta \psi\varphi $ implies that $ \zeta^{p} = \det(\varphi\psi\varphi^{-1}\psi^{-1}) = 1 $. Since $ \zeta \neq 1 $, it is a primitive $ p $-th root of unity.
	
	Choose a $c_1$-eigenvector $e_0$ of $\varphi$. The vectors $e_0,\psi(e_0),\ldots,\psi^{p-1}(e_0)$ belong to the distinct eigenspaces of $\varphi$ with eigenvalues $c_1,c_1\zeta,\ldots,c_1\zeta^{p-1}$, and hence form a basis of $V$. Since $\psi^p(e_0)=a e_0$ for some $a\in\KK^\times$, choose $c_2\in\KK^\times$ with $c_2^p=a$ and set $e_i=c_2^{-i}\psi^i(e_0)$. Then
	\begin{equation*}
		\varphi(e_i)=c_1\zeta^i e_i,\qquad \psi(e_i)=c_2e_{i+1},\qquad e_p=e_0.
	\end{equation*}
	The final assertion follows immediately.
\end{proof}

\begin{proof}[Proof of Theorem \ref{main theorem: subgroup criterion of liftability when N=p}]
	Only the sufficiency requires proof. Suppose, to the contrary, that $ G $ is not liftable. By Lemma \ref{lemma: Sylow criterion of liftability}, there exists a Sylow $ p $-subgroup $ G_{p} $ of $ G $ that is not liftable. Proposition \ref{proposition: liftability of p-groups when N=p} then gives an abelian subgroup $ H $ of $ G_{p} $ such that $ \pi^{-1}(H) $ is not abelian. Choose $ \varphi,\psi \in \pi^{-1}(H) $ that do not commute. Their projective images commute, so there exists $ \zeta \neq 1 $ such that $ \varphi\psi = \zeta\psi\varphi $. By Lemma \ref{lemma: structure of semi-commutative pair when N=p}, we have $ \pi(\varphi)^{p}=\pi(\psi)^{p}=1 $. The nontrivial scalar commutator implies that neither projective image is trivial and that neither is a power of the other. Hence
	\begin{equation*}
		\langle\pi(\varphi),\pi(\psi)\rangle\simeq(\ZZ/p\ZZ)^2.
	\end{equation*}
	The same scalar commutator shows that this subgroup is not liftable, contradicting the assumption.
\end{proof}

The preceding proof also gives the following concrete criterion.

\begin{cor}
		\label{corollary: structure of unliftable G when N=p}
	Let $G$ be a finite subgroup of $\PGL(V)$, and fix a primitive $p$-th root of unity $\zeta$. Then $G$ is not liftable if and only if there exist $\varphi,\psi\in\pi^{-1}(G)$ such that $\varphi^p=\psi^p=\I_V$ and $\varphi\psi=\zeta\psi\varphi$.
\end{cor}

\begin{proof}
	If $G$ is not liftable, the proof of Theorem \ref{main theorem: subgroup criterion of liftability when N=p} produces two projective elements of order $p$ with lifts having a nontrivial scalar commutator. Rescaling these lifts makes their $p$-th powers equal to $\I_V$, and replacing one projective generator by a suitable nonzero power changes the commutator to the prescribed root $\zeta$. Conversely, the displayed relation is unchanged by scalar rescaling, so the abelian projective subgroup generated by $\pi(\varphi)$ and $\pi(\psi)$ cannot admit commuting lifts. Hence neither that subgroup nor $G$ is liftable.
\end{proof}

\begin{prop}
	\label{proposition: F-liftability when N=p}
	A $p$-subgroup $G \subseteq \Lin(F)$ is $ F $-liftable if and only if $G$ is liftable and all of its elements are $F$-liftable.
\end{prop}

\begin{proof}
	Only the sufficiency requires proof. Since $ G $ is liftable, Proposition \ref{proposition: liftability of p-groups when N=p} shows that $ \pi^{-1}(G) $ is abelian. Hence $G$ is abelian. By the structure theorem of finite abelian groups, there exist $ g_{1},\cdots,g_{k} \in G $ such that
	\begin{equation*}
		G = \langle g_{1} \rangle \times \cdots \times \langle g_{k} \rangle.
	\end{equation*}
	
	Choose an $ F $-lifting $ \widetilde{g_{i}} $ for each $ g_{i} $. These matrices commute because $\pi^{-1}(G)$ is abelian. Then the group
	\begin{equation*}
		\widetilde{G} = \langle \widetilde{g_{1}},\ldots,\widetilde{g_{k}} \rangle
	\end{equation*}
	is an $ F $-preserving group that maps isomorphically onto $G$, by the direct-product decomposition and the order-preserving property of each $F$-lifting. Thus $ G $ is $ F $-liftable.
\end{proof}

\begin{lem}
	\label{lemma: structure of non-F-liftable element}
	Let $G$ be a finite subgroup of $\Lin(F)$. If some element of $G$ is not $F$-liftable, then $p\mid d$ and there exist a primitive $p$-th root of unity $\zeta$, an element $\varphi\in\pi^{-1}(G)$, and a basis $e_0,\ldots,e_{p-1}$ of $V$ such that
	\begin{equation*}
		F\circ\varphi=\zeta F,\qquad \varphi(e_i)=\zeta^i e_i.
	\end{equation*}
	In particular, $\varphi^p=\I_V$, and $\pi(\varphi)$ is an element of order $p$ that is not $F$-liftable.
\end{lem}

\begin{proof}
	Let $g\in G$ be an element that is not $F$-liftable and set $C=\langle g\rangle$. Since $\dim V=p$, the $F$-liftability part of Lemma \ref{lemma: Sylow criterion of liftability} implies that $p\mid d$ and that the Sylow $p$-subgroup $C_p$ of $C$ is not $F$-liftable. Since $C_p$ is cyclic, a generator of $C_p$ is not $F$-liftable. We may therefore choose a non-$F$-liftable $p$-element $g_0\in G$ of minimal order. Then $g_0^p$ is $F$-liftable.

	Apply Lemma \ref{lemma: existence of F-preserving p-preimages for single elements} to $g_0$. Its integer $c$ equals $1$, so there exists $\widetilde g_0\in\pi^{-1}(g_0)$ of order $p^s$, where $s=\nu_p(d)+1$, such that $ F \circ \widetilde{g_0}=F$. Lemma \ref{lemma: computation of the order of 1-d wrt p^r} and direct calculation for $d=2$ gives $\ell_{p,d}(s)=p$ for $d\ge 3$. In the normal form of Lemma \ref{lemma: normal form of F-preserving elements}, at least one root is primitive of order $p^s$ because $\widetilde g_0$ has that order. Its block already has length $p=\dim V$, so it is the only block. Thus there are a primitive $p^s$-th root of unity $\xi$ and a basis $e_0,\ldots,e_{p-1}$ such that
	\begin{equation*}
		\widetilde g_0(e_i)=\xi^{(1-d)^i}e_i.
	\end{equation*}
	Set $\zeta=\xi^{-d}$ and $\varphi=\xi^{-1}\widetilde g_0$. Since $s=\nu_p(d)+1$, the root $\zeta$ is a primitive $p$-th root of unity and
	\begin{equation*}
		(1-d)^i\equiv 1-id\pmod{p^s}.
	\end{equation*}
	It follows that $\varphi(e_i)=\zeta^i e_i$ and $F\circ\varphi=\zeta F$. Finally, any order-$p$ preimage of $\pi(\varphi)$ is $\alpha\varphi$ for some $\alpha$ with $\alpha^p=1$. Since $p\mid d$, it still acts on $F$ by the nontrivial multiplier $\zeta$. Thus $\pi(\varphi)$ is not $F$-liftable.
\end{proof}

\begin{proof}[Proof of Theorem \ref{main theorem: subgroup criterion for F-liftability when N=p}]
	Only the sufficiency requires proof. If $p\nmid d$, the conclusion follows immediately from Lemma \ref{lemma: Sylow criterion of liftability}. Assume that $p\mid d$. By the same lemma, it suffices to prove that a Sylow $p$-subgroup $G_p$ of $G$ is $F$-liftable. Theorem \ref{main theorem: subgroup criterion of liftability when N=p} shows that $G_p$ is liftable. If some element of $G_p$ were not $F$-liftable, Lemma \ref{lemma: structure of non-F-liftable element} would produce a non-$F$-liftable element of order $p$ in $G_p$, contrary to the hypothesis. Thus every element of $G_p$ is $F$-liftable, and Proposition \ref{proposition: F-liftability when N=p} implies that $G_p$ is $F$-liftable.
\end{proof}

\begin{rmk}
	\label{remark: structure of non-liftable G}
	Let $G$ be a finite subgroup of $\Lin(F)$ that is not $F$-liftable. Then $p\mid d$, and the following alternatives describe the two possible obstructions; compare \cite[Lemmas 4.13 and 4.14]{oguiso2019quintic}.
	\begin{enumerate}
		\item If not all elements of $G$ are $F$-liftable, then there exist a primitive $p$-th root of unity $\zeta$, an element $\varphi\in\pi^{-1}(G)$, and a basis $e_0,\ldots,e_{p-1}$ such that $F\circ\varphi=\zeta F$ and $\varphi(e_i)=\zeta^i e_i$. Consequently, with respect to the dual basis $x_0,\ldots,x_{p-1}$, the form $F$ is a linear combination of
		\begin{equation*}
			\{\prod\limits_{i=0}^{p-1}x_i^{\alpha_i} \mid \sum\limits_{i=0}^{p-1}\alpha_{i}=d,\quad\sum\limits_{i=0}^{p-1}i\alpha_{i} \equiv 1 \pmod{p}\}.
		\end{equation*}
		\item If every element of $G$ is $F$-liftable, then there exist a primitive $p$-th root of unity $\zeta$ and elements $\varphi,\psi\in\pi^{-1}(G)$ such that $\varphi\psi=\zeta\psi\varphi$, $F\circ\varphi=F\circ\psi=F$, and, with respect to a suitable basis $e_0,\ldots,e_{p-1}$,
		\begin{equation*}
			\varphi(e_{i}) = \zeta^{i}e_{i},\quad \psi(e_{i}) = e_{i+1}, \quad e_{p}=e_{0}.
		\end{equation*} 
	Consequently, with respect to the dual basis $x_0,\ldots,x_{p-1}$, the form $F$ is a linear combination of
	\begin{equation*}
	\{\sum\limits_{j=0}^{p-1}\prod_{i=0}^{p-1}x_{i}^{\alpha_{i+j}} \mid \sum\limits_{i=0}^{p-1}\alpha_{i}=d,\quad\sum\limits_{i=0}^{p-1}i\alpha_{i} \equiv 0 \pmod{p}\},
	\end{equation*}
	where the indices are read modulo $p$.
	\end{enumerate}
\end{rmk}

\section{The Case $N=2p$ and $d=p$}
\label{section: N=2p and d=p}

\subsection{The General Odd-Prime Case}
Having settled the case $N=p$, we now turn to the more involved case $N=2p$ and $d=p$, where $p$ is an odd prime. By Theorem \ref{theorem: order of elements in Lin(F)}, all elements of $p$-power order in $\Lin(F)$ have order at most $p^2$, and every such element that is not $F$-liftable has order $p$.

We first study forms that split into two disjoint sets of $p$ variables.

\begin{defn}
	\label{definition: split of F}
	Let $ V = V_{1} \oplus V_{2} $ be a direct sum decomposition of $ V $ with $ \dim V_{1} = \dim V_{2} = p $. For $i=1,2$, denote by $ \operatorname{pr}_{i} \colon V \rightarrow V_{i} $ the projection with respect to this decomposition, by $\pi_i\colon\GL(V_i)\to\PGL(V_i)$ the canonical projection, and by $ F_{i} \in S^{p}(V_{i}^{*}) $ the restriction of $ F $ to $ V_{i} $. Note that $ \operatorname{pr}_{i} $ induces a pull-back map $ \operatorname{pr}_{i}^{*} \colon S^{p}(V_{i}^{*}) \rightarrow S^{p}(V^{*}) $. We call $ V = V_{1} \oplus V_{2} $ a split of $ F $ if 
	\begin{equation*}
		F = \operatorname{pr}_{1}^{*}F_{1} + \operatorname{pr}_{2}^{*}F_{2}.
	\end{equation*}
\end{defn}

In coordinates, there exist bases $ e_{0},\cdots,e_{p-1} $ of $ V_{1} $ and $ e_{p},\cdots,e_{2p-1} $ of $ V_{2} $ such that $ F $ has the form
\begin{equation*}
	F = f_{1}(x_{0},\cdots,x_{p-1})+f_{2}(x_{p},\cdots,x_{2p-1}),
\end{equation*}
where $ f_{i} $ is the coordinate polynomial of $ F_{i} $ and $ x_{0},\cdots,x_{2p-1} $ is the dual basis of $ e_{0},\cdots,e_{2p-1} $. Since $ F $ is nonsingular, this coordinate expression implies that both $ F_{1} $ and $ F_{2} $ are nonsingular. 

\begin{lem}
	\label{lemma: eigenspace decomposition induce split}
	Let $F \in S^{p}(V^{*})$ be a nonsingular homogeneous form of degree $p$. Let $ \varphi \in \pi^{-1}(\Lin(F)) $ be an element of order $ p $. Suppose that the eigenspace decomposition of $ \varphi $ has the form $ V = V_{1} \oplus V_{2} $, where $ \dim V_{1} = \dim V_{2} = p $. Then $ V_{1} \oplus V_{2} $ is a split of $ F $.
\end{lem}

\begin{proof}
	Choose coordinates $ x_{0},\cdots,x_{2p-1} $ of $ V $ such that $ \varphi $ is represented by the diagonal matrix $ \diag(\zeta_{0},\cdots,\zeta_{2p-1}) $ with respect to the chosen coordinates. Then we may assume
	\begin{equation*}
		\zeta_{0}=\cdots=\zeta_{p-1},\quad 	\zeta_{p}=\cdots=\zeta_{2p-1},\quad \zeta_{0} \neq \zeta_{p}. 
	\end{equation*}
	
	Write $F\circ\varphi=\lambda F$. Nonsingularity implies that, for every coordinate $x_i$, the form $F$ contains a monomial $x_i^{p-1}x_j$ for some $j$; otherwise the corresponding coordinate point would be singular. In particular, $F$ contains monomials $x_0^{p-1}x_{j_1}$ and $x_p^{p-1}x_{j_2}$ for suitable $j_1,j_2$. Hence
	\begin{equation*}
		\lambda=\zeta_0^{p-1}\zeta_{j_1}=\zeta_p^{p-1}\zeta_{j_2}.
	\end{equation*}
	Each expression on the right belongs respectively to $\{1,\zeta_p/\zeta_0\}$ and $\{1,\zeta_0/\zeta_p\}$. Since $p$ is odd and $\zeta_0\neq\zeta_p$, their only common value is $1$. Thus $\lambda=1$, $\zeta_{j_1}=\zeta_0$, and $\zeta_{j_2}=\zeta_p$. Moreover, $\zeta_0^{p-k}\zeta_p^k\neq1$ for $1\leq k\leq p-1$, so no monomial of $F$ involves variables from both eigenspaces. Therefore
	\begin{equation*}
		F = f_{1}(x_{0},\cdots,x_{p-1})+f_{2}(x_{p},\cdots,x_{2p-1}),
	\end{equation*}
	where each $f_i$ is a nonsingular homogeneous polynomial of degree $p$. Hence $F$ splits.
\end{proof}

\begin{lem}
	\label{lemma: liftability for split F}
	Suppose that $ F $ admits a split $ V_{1}\oplus V_{2} $. If all abelian $ p $-subgroups of $ \Lin(F) $ of order $ \le p^{3} $ are liftable, then either $ \Lin(F_{1}) $ or $ \Lin(F_{2}) $ is liftable. 
\end{lem}

\begin{proof}
	Suppose, to the contrary, that neither $\Lin(F_1)$ nor $\Lin(F_2)$ is liftable. Fix a primitive $p$-th root of unity $\zeta$. By Corollary \ref{corollary: structure of unliftable G when N=p}, for $i=1,2$ there exist $\varphi_i,\psi_i\in\pi_i^{-1}(\Lin(F_i))$ such that
	\begin{equation*}
		\varphi_i^p=\psi_i^p=\I_{V_i},\qquad [\varphi_i,\psi_i]=\zeta\I_{V_i}.
	\end{equation*}
	Set $E_i=\langle\pi_i(\varphi_i),\pi_i(\psi_i)\rangle\simeq(\ZZ/p\ZZ)^2$. The kernel of $\langle\varphi_i,\psi_i\rangle\to E_i$ is generated by $\zeta\I_{V_i}$, which acts trivially on the degree-$p$ form $F_i$. Consequently, the multipliers by which the matrices act on $F_i$ descend to a character $\chi_i\colon E_i\to\mu_p$. The scalar commutator defines a nondegenerate alternating pairing on $E_i$. Choose a nonzero vector in $\ker(\chi_i)$ as the second basis vector, and choose the first so that their commutator is $\zeta$. Because $p$ is odd, any product of the original order-$p$ lifts again has $p$-th power $\I_{V_i}$. Relabeling lifts of this new basis as $\varphi_i,\psi_i$, we retain the displayed relations and have $F_i\circ\psi_i=F_i$. Write
	\begin{equation*}
		F_{i} \circ \varphi_{i} = \zeta^{\alpha_{i}} F_{i},\qquad \alpha_i\in\{0,\ldots,p-1\}.
	\end{equation*}
	Let $\xi$ be a primitive $p^2$-th root of unity satisfying $\xi^p=\zeta$, and define $\varphi,\psi\in\GL(V)$ by
	\begin{equation*}
		\varphi|_{V_{i}} = \xi^{-\alpha_{i}}\varphi_{i}, \quad  \psi|_{V_{i}} = \psi_{i}.
	\end{equation*}
	Then $F\circ\varphi=F\circ\psi=F$, $\varphi^{p^2}=\psi^p=\I_V$, and $[\varphi,\psi]=\zeta\I_V$. Thus $\pi(\varphi)$ and $\pi(\psi)$ commute and have orders at most $p^2$ and $p$, respectively. They generate an abelian $p$-subgroup $H\subseteq\Lin(F)$ of order at most $p^3$. The scalar commutator is unchanged if either matrix is multiplied by a scalar, so no lifting of $H$ can have commuting preimages of these two generators. Hence $H$ is not liftable, a contradiction.
\end{proof}

\begin{thm}
	\label{main theorem: subgroup criterion of liftability when N=2p}
	Let $p \geq 3$ be a prime, let $V$ be a $2p$-dimensional $\KK$-vector space, and let $F \in S^{p}(V^{*})$ be a nonsingular homogeneous form of degree $p$. Then $\Lin(F)$ is liftable if and only if all its abelian $p$-subgroups of order at most $p^{3}$ are liftable. In particular, $\Lin(F)$ is liftable if and only if it satisfies the property $\mathrm{L}(p,3)$.
\end{thm}

\begin{proof}
	Only the sufficiency requires proof. Since $2\nmid\gcd(2p,p)$, Corollary \ref{corollary: liftability of p groups when gcd(d,N,p)=1}(2) shows that every $2$-subgroup of $\Lin(F)$ is $F$-liftable, hence liftable. By Lemma \ref{lemma: Sylow criterion of liftability}, it therefore suffices to prove that a Sylow $ p $-subgroup $ G $ of $ \Lin(F) $ is liftable. If $G$ is trivial, there is nothing to prove.

	Choose an element $g$ of order $ p $ in the center of $G$. For every $ h \in G $, the order bound gives $\ord(h)\leq p^2$, so the abelian subgroup $ \langle g,h \rangle $ has order at most $p^3$ and is liftable by hypothesis. Pick a lifting $\widetilde{g}$ of $g$ and consider its eigenspace decomposition:
	\begin{equation*}
		V = V_{1} \oplus \cdots \oplus V_{k}.
	\end{equation*} 
	By Lemma \ref{lemma: the space decomposition}, all $ V_{i} $ are $\pi^{-1}(G)$-invariant. Define $ G_{i} \subseteq \PGL(V_{i}) $ as in Lemma \ref{lemma: liftability and decompositions}.
	
	If $p\nmid\dim V_i$ for some $i$, Corollary \ref{corollary: liftability of p groups when gcd(d,N,p)=1}(1) shows that $G_i$ is liftable. Lemma \ref{lemma: liftability and decompositions} then implies that $G$ is liftable.
	
	Otherwise, $k=2$ and $\dim V_1=\dim V_2=p$ because $g$ is nontrivial. Lemma \ref{lemma: eigenspace decomposition induce split} shows that $F$ splits, and Lemma \ref{lemma: F-liftability and decompositions} gives $G_i\subseteq\Lin(F_i)$ for $i=1,2$. By Lemma \ref{lemma: liftability for split F}, one of $\Lin(F_1)$ and $\Lin(F_2)$ is liftable. The corresponding subgroup $G_i$ is therefore liftable, and Lemma \ref{lemma: liftability and decompositions} implies that $G$ is liftable.
\end{proof}

\begin{lem}
	\label{lemma: F-liftability for split F}
	Suppose that $ F $ admits a split $ V_{1} \oplus V_{2} $ and that $ \Lin(F_{1}) $ is liftable but not $ F_{1} $-liftable. Then:
	\begin{enumerate}
		\item If not all elements of $ \Lin(F_{2}) $ are $ F_{2} $-liftable, then $ \Lin(F) $ has an element that is not $F$-liftable. 
		\item If all elements of $\Lin(F_2)$ are $F_2$-liftable, but $\Lin(F_2)$ is not $F_2$-liftable, then $\Lin(F)$ has a $p$-subgroup of order at most $p^3$ that is not $F$-liftable.
	\end{enumerate}
\end{lem}

\begin{proof}
	Since $\Lin(F_1)$ is not $F_1$-liftable, Lemma \ref{lemma: Sylow criterion of liftability} provides a Sylow $p$-subgroup $P_1\subseteq\Lin(F_1)$ that is not $F_1$-liftable. On the other hand, $P_1$ is liftable because $\Lin(F_1)$ is liftable. Proposition \ref{proposition: F-liftability when N=p} therefore shows that $P_1$ contains an element that is not $F_1$-liftable.

	Fix a primitive $p$-th root of unity $\zeta\in\KK$, and choose a primitive $p^{2}$-th root of unity $\xi\in\KK$ with $\xi^{p} = \zeta$. By Lemma \ref{lemma: structure of non-F-liftable element}, after replacing the resulting matrix by a suitable power, there exists $\varphi_1\in\pi_1^{-1}(P_1)$ such that $\varphi_1^p=\I_{V_1}$ and $F_1\circ\varphi_1=\zeta F_1$.
	
	(1) If not all elements of $\Lin(F_2)$ are $F_2$-liftable, Lemma \ref{lemma: structure of non-F-liftable element} gives, after the same normalization, a matrix $\varphi_2\in\pi_2^{-1}(\Lin(F_2))$ such that $\varphi_2^p=\I_{V_2}$ and $F_2\circ\varphi_2=\zeta F_2$. Define $\varphi\in\GL(V)$ by
	\begin{equation*}
		\varphi|_{V_{i}} = \varphi_{i},\quad i=1,2.
	\end{equation*}
	
	Then
	\begin{equation*}
		F \circ \varphi  = \zeta F.
	\end{equation*}
	
	Since $\varphi^p=\I_V$ and the multiplier $\zeta$ is nontrivial, $\pi(\varphi)$ has order $p$ and is not $F$-liftable.
	
	(2) If all elements of $\Lin(F_2)$ are $F_2$-liftable but $\Lin(F_2)$ is not $F_2$-liftable, Remark \ref{remark: structure of non-liftable G}(2), followed if necessary by replacing one matrix by a suitable power, gives $\varphi_2,\psi_2\in\pi_2^{-1}(\Lin(F_2))$ such that $[\varphi_2,\psi_2]=\zeta\I_{V_2}$, $\varphi_2^p=\psi_2^p=\I_{V_2}$, and $F_2\circ\varphi_2=F_2\circ\psi_2=F_2$. Define $\varphi,\psi\in\GL(V)$ by
	\begin{equation*}
		\begin{aligned}
			&\varphi|_{V_{1}} = \xi^{-1} \varphi_{1}, \varphi|_{V_{2}} =  \varphi_{2};\\ &\psi|_{V_{1}} = \I_{V_{1}}, \psi|_{V_{2}} = \psi_{2}. 
		\end{aligned}
	\end{equation*}
	Then $ F\circ \varphi = F \circ \psi = F $ and $ \varphi^{p^{2}} = \psi^{p} = \I_{V} $. Furthermore, letting $ \eta = [\varphi,\psi] $, we have
	\begin{equation*}
		\eta|_{V_{1}} = \I_{V_{1}}, \quad \eta|_{V_{2}} = \zeta \I_{V_{2}}.
	\end{equation*}
	Thus
	\begin{equation*}
		\eta^{p} = \I_{V},\quad [\varphi,\eta] = [\psi,\eta] = \I_{V}, \quad \varphi^{p} = \zeta^{-1}\eta.
	\end{equation*}
	
	Let $G$ be the subgroup of $\Lin(F)$ generated by $\pi(\varphi)$ and $\pi(\psi)$. Let $ H $ be the normal subgroup of $ G $ generated by $ \pi(\eta) $. Then $ |H| = p $, while $G/H$ is abelian and generated by two elements of order at most $p$. Thus $|G/H|\leq p^2$ and $|G| \leq p^{3}$. 
	
	If $G$ had an $F$-lifting $\widetilde{G}$, let $\widetilde{\varphi},\widetilde{\psi},\widetilde{\eta}$ be the elements corresponding to $\pi(\varphi),\pi(\psi),\pi(\eta)$, respectively. Then $ \widetilde{\eta} = [\widetilde{\varphi},\widetilde{\psi}] = [\varphi,\psi] = \eta$. Moreover, $ \widetilde{\varphi} $ and $\varphi$ differ by a $p$-th root of unity because both preserve $F$, and hence $ \varphi^{p} = \widetilde{\varphi}^{p} $. It follows that $\zeta\I_{V} = \eta\varphi^{-p} = \widetilde{\eta}\widetilde{\varphi}^{-p}$ belongs to $\widetilde G$, contradicting the fact that a lifting contains no nontrivial scalar. Thus $G$ is a non-$F$-liftable $p$-subgroup of order at most $p^3$. 
\end{proof}

\begin{cor}
	\label{corollary: F-liftability for split F}
	Suppose that $ F $ admits a split $ V_{1}\oplus V_{2} $. If all $ p $-subgroups of $ \Lin(F) $ of order $ \le p^{3} $ are $ F $-liftable, then either $ \Lin(F_{1}) $ is $ F_{1} $-liftable, or $ \Lin(F_{2}) $ is $ F_{2} $-liftable. 
\end{cor}

\begin{proof}
	Suppose that neither $\Lin(F_1)$ nor $\Lin(F_2)$ is $F_i$-liftable. The hypothesis implies in particular that every abelian $p$-subgroup of $\Lin(F)$ of order at most $p^3$ is liftable. Lemma \ref{lemma: liftability for split F} therefore shows that at least one of $\Lin(F_1)$ and $\Lin(F_2)$ is liftable; after interchanging the indices, assume that $\Lin(F_1)$ is liftable but not $F_1$-liftable. If $\Lin(F_2)$ contains an element that is not $F_2$-liftable, Lemma \ref{lemma: F-liftability for split F}(1) produces a non-$F$-liftable element of order $p$ in $\Lin(F)$, contrary to the hypothesis. Otherwise every element of $\Lin(F_2)$ is $F_2$-liftable, and Lemma \ref{lemma: F-liftability for split F}(2) produces a non-$F$-liftable $p$-subgroup of order at most $p^3$, again a contradiction.
\end{proof}

\begin{thm}
	\label{main theorem: subgroup criterion of F-liftability when N=2p}
	Let $p \geq 3$ be a prime, let $V$ be a $2p$-dimensional $\KK$-vector space, and let $F \in S^{p}(V^{*})$ be a nonsingular homogeneous form of degree $p$. Then $\Lin(F)$ is $F$-liftable if and only if all its $p$-subgroups of order at most $p^{3}$ are $F$-liftable, or equivalently, $\Lin(F)$ satisfies the property $\mathrm{FL}(p,3)$.
\end{thm}

\begin{proof}
	Only the sufficiency requires proof. By Lemma \ref{lemma: Sylow criterion of liftability}, it suffices to prove that a Sylow $ p $-subgroup $ G $ of $ \Lin(F) $ is $ F $-liftable. 
	If $G$ is trivial, there is nothing to prove.

	Choose an element $g$ of order $ p $ in the center of $G$. For every $h\in G$, the order bound gives $\ord(h)\leq p^2$, so $\langle g,h\rangle$ is an abelian subgroup of order at most $p^3$ and is liftable by hypothesis. The same hypothesis applied to $\langle g\rangle$ gives an $F$-lifting $\widetilde g$ of $g$. Consider its eigenspace decomposition:
	\begin{equation*}
		V = V_{1} \oplus \cdots \oplus V_{k}.
	\end{equation*} 
	By Lemma \ref{lemma: the space decomposition}, all $ V_{i} $ are $\pi^{-1}(G)$-invariant. Define $ G_{i} \subseteq \PGL(V_{i}) $ as in Lemma \ref{lemma: liftability and decompositions} and let $ F_{i} $ be the restriction of $ F $ to $ V_{i} $.
	
	If $p\nmid\dim V_i$ for some $i$, Lemma \ref{lemma: decomposition of F} shows that $F_i$ is nonsingular, and Lemma \ref{lemma: F-liftability and decompositions} gives $G_i\subseteq\Lin(F_i)$. Corollary \ref{corollary: liftability of p groups when gcd(d,N,p)=1}(2) then shows that $G_i$ is $F_i$-liftable, and Lemma \ref{lemma: F-liftability and decompositions} implies that $G$ is $F$-liftable.
	
	Otherwise, $ k=2 $ and $ \dim V_{1} = \dim V_{2} = p $ because $ g $ is nontrivial. Lemma \ref{lemma: eigenspace decomposition induce split} shows that $F$ splits, and Lemma \ref{lemma: F-liftability and decompositions} gives $G_i\subseteq\Lin(F_i)$ for $i=1,2$. Corollary \ref{corollary: F-liftability for split F} shows that one of $\Lin(F_1)$ and $\Lin(F_2)$ is $F_i$-liftable, so the corresponding $G_i$ is $F_i$-liftable. Lemma \ref{lemma: F-liftability and decompositions} now implies that $G$ is $F$-liftable.
\end{proof}

\subsection{Cubic Forms in Six Variables}
The bounds in Theorem \ref{main theorem: subgroup criterion of liftability when N=2p} and Theorem \ref{main theorem: subgroup criterion of F-liftability when N=2p} are not optimal for $p=3$. In this subsection we prove Theorem \ref{main theorem: cubic fourfold criteria}, lowering the order bound from $p^{3}=27$ to $p^{2}=9$. The following lemma captures the special feature of this case. Throughout the subsection, fix a primitive ninth root of unity $\xi\in\KK$ and set $\zeta=\xi^3$, a primitive cube root of unity.

\begin{lem}
		\label{lemma: special case p=3}
	Let $V$ be a $3$-dimensional $\KK$-vector space and let $F\in S^3(V^*)$ be nonsingular. If not all elements of $\Lin(F)$ are $F$-liftable, then there exist $\varphi,\psi\in\GL(V)$ of order $3$ such that $F\circ\varphi=F\circ\psi=F$ and $\psi\varphi=\zeta\varphi\psi$.
\end{lem}

\begin{proof}
	By Lemma \ref{lemma: structure of non-F-liftable element}, after replacing the resulting matrix by its square if necessary, there exists $A\in\GL(V)$ such that $F\circ A=\zeta F$ and there is a basis $\{e_0,e_1,e_2\}$ of $V$ such that
	\begin{equation*}
		A(e_{i}) = \zeta^{i}e_{i}
	\end{equation*} 
	
	Let $\{x_0,x_1,x_2\}$ be the dual basis of $\{e_0,e_1,e_2\}$. Since $x_j\circ A=\zeta^j x_j$ and $F\circ A=\zeta F$, the form $F$ must be
	\begin{equation*}
		F=\lambda_0x_0^2x_1+\lambda_1x_1^2x_2+\lambda_2x_2^2x_0.
	\end{equation*}
	Nonsingularity implies that all $\lambda_j$ are nonzero. The map
	\begin{equation*}
		(a_0,a_1,a_2)\longmapsto(a_0^2a_1,a_1^2a_2,a_2^2a_0)
	\end{equation*}
	is an isogeny of $(\KK^\times)^3$: its exponent matrix has determinant $9$. It is therefore surjective because $\KK$ is algebraically closed. After a diagonal change of basis, we may consequently assume that $\lambda_0=\lambda_1=\lambda_2=1$. For $i=0,1,2$, set
	\begin{equation*}
	v_{i} = e_{0} + \zeta^{i}e_{1}+\zeta^{2i}e_{2}.
	\end{equation*}
	
	Let $\{y_0,y_1,y_2\}$ be the dual basis of $\{v_0,v_1,v_2\}$. Then
	\begin{equation*}
		x_{i} = y_{0} + \zeta^{i}y_{1}+\zeta^{2i}y_{2}.
	\end{equation*}
	
	A direct computation gives
	\begin{equation*}
		F = x_{0}^{2}x_{1}+x_{1}^{2}x_{2}+x_{2}^{2}x_{0} = 3y_{0}^{3} + 3\zeta y_{1}^{3} +3\zeta^{2} y_{2}^{3}.
	\end{equation*}
	
	Consider the following linear maps $\varphi,\psi \in \GL(V)$:
	\begin{equation*}
		\begin{aligned}
			&\varphi(v_{0})=v_{0},\ \varphi(v_{1})=\zeta v_{1},\ \varphi(v_{2})=\zeta^{2}v_{2};\\
			&\psi(v_{0})=\xi^{-2} v_{2},\ \psi(v_{1})= \xi v_{0},\ \psi(v_{2})=\xi v_{1}.\\
		\end{aligned}
	\end{equation*}
	
	A direct verification gives $\varphi^3=\psi^3=\I_V$, $\psi\varphi=\zeta\varphi\psi$, and $F\circ\varphi=F\circ\psi=F$.
\end{proof}

\begin{prop}
	\label{proposition: liftability criterion when N=d=3}
	Let $V$ be a $3$-dimensional $\KK$-vector space and let $F\in S^3(V^*)$ be nonsingular. Then the following conditions are equivalent:
	\begin{enumerate}
		\item $ \Lin(F) $ is not liftable.
		\item $ \Lin(F) $ is not $ F $-liftable.
		\item There exist $\varphi,\psi\in\GL(V)$ of order $3$ such that $F \circ \varphi = F \circ \psi =F$ and $\psi\varphi = \zeta \varphi\psi$.
	\end{enumerate} 
\end{prop}

\begin{proof}
		Condition (3) implies (1), because the projective images of $\varphi$ and $\psi$ generate a nonliftable subgroup isomorphic to $(\ZZ/3\ZZ)^2$, and (1) implies (2). It remains to prove that (2) implies (3). Suppose, to the contrary, that (3) fails. By Lemma \ref{lemma: special case p=3}, every element of $\Lin(F)$ is $F$-liftable. Since $\Lin(F)$ is not $F$-liftable, Theorem \ref{main theorem: subgroup criterion for F-liftability when N=p} yields a nonliftable subgroup isomorphic to $(\ZZ/3\ZZ)^2$. Corollary \ref{corollary: structure of unliftable G when N=p} then gives order-three lifts $\varphi,\psi$ satisfying $\psi\varphi=\zeta\varphi\psi$. Each differs from an $F$-lifting of its projective image by a third root of unity, which acts trivially on the cubic form. Hence $F\circ\varphi=F\circ\psi=F$, a contradiction.
\end{proof}

For the remainder of this subsection, let $V$ be a $6$-dimensional $\KK$-vector space and let $F\in S^3(V^*)$ be nonsingular.

\begin{lem}
	\label{lemma: F-liftability for split F when p=3}
	Suppose that $F$ admits a split $V_1\oplus V_2$. If all abelian $3$-subgroups of $\Lin(F)$ of order at most $9$ are liftable, then either $\Lin(F_1)$ is $F_1$-liftable or $\Lin(F_2)$ is $F_2$-liftable.
\end{lem}

\begin{proof}
	Suppose, to the contrary, that $ \Lin(F_{i}) $ is not $ F_{i} $-liftable for $ i=1,2 $. By Proposition \ref{proposition: liftability criterion when N=d=3}, there exist $\varphi_{i},\psi_{i}\in\GL(V_{i})$ of order $3$ such that $F_{i} \circ \varphi_{i} = F_{i} \circ \psi_{i} =F_{i}$ and $\psi_{i}\varphi_{i} = \zeta \varphi_{i}\psi_{i}$. Define $\varphi,\psi\in\GL(V)$ by
	\begin{equation*}
		\begin{aligned}
			&\varphi|_{V_{1}} =  \varphi_{1}, \varphi|_{V_{2}} = \varphi_{2};\\ &\psi|_{V_{1}} =  \psi_{1}, \psi|_{V_{2}} = \psi_{2}. 
		\end{aligned}
	\end{equation*}
	
	Then $F\circ\varphi=F\circ\psi=F$, $\varphi^3=\psi^3=\I_V$, and $\psi\varphi=\zeta\varphi\psi$. Their projective images generate a nonliftable subgroup isomorphic to $(\ZZ/3\ZZ)^2$, contradicting the hypothesis.
\end{proof}

\begin{proof}[Proof of Theorem \ref{main theorem: cubic fourfold criteria}(1)]
	Only the sufficiency requires proof. Suppose that every $3$-subgroup of order at most $9$ is liftable but $\Lin(F)$ is not. By Theorem \ref{main theorem: subgroup criterion of liftability when N=2p}, there is a nonliftable abelian $3$-subgroup $H\subseteq\Lin(F)$ of order at most $27$. The hypothesis gives $|H|=27$, and Theorem \ref{theorem: order of elements in Lin(F)} excludes elements of order $27$.

	We claim that $H$ is not isomorphic to $(\ZZ/3\ZZ)^3$. Indeed, in that case any two elements of $H$ generate a subgroup $L$ of order at most $9$ and hence a liftable subgroup. If $\widetilde L$ is a lifting of $L$, then $\pi^{-1}(L)=\KK^\times\widetilde L$ is abelian. Thus arbitrary lifts of any two elements of $H$ commute, and $\pi^{-1}(H)$ is abelian. Commuting order-three lifts of a basis of $H$ would then generate a lifting of $H$, a contradiction. Therefore
	\begin{equation*}
		H\simeq\ZZ/9\ZZ\times\ZZ/3\ZZ.
	\end{equation*}
		Choose generators $a,b$ of orders $9$ and $3$. By Theorem \ref{theorem: order of elements in Lin(F)}, the element $a$ is $F$-liftable; choose an order-nine lift $\varphi$ such that $F\circ\varphi=F$. The subgroup $\langle b\rangle$ is liftable by the hypothesis, so choose an order-three lift $\psi$ of $b$. If these lifts commuted, they would generate a lifting of $H$; hence their scalar commutator is nontrivial. Since $\psi^3=\I_V$, this commutator is a primitive cube root of unity. Replacing $a$ and $\varphi$ by suitable powers coprime to $9$, if necessary, gives
	\begin{equation*}
		\varphi\psi=\zeta\psi\varphi,\qquad \ord(\varphi)=9,\qquad \ord(\psi)=3.
	\end{equation*}

	Choose a Sylow $3$-subgroup $G$ of $\Lin(F)$ containing $H$, and let $V=V_1\oplus\cdots\oplus V_k$ be an irreducible decomposition as a $\pi^{-1}(G)$-representation. Denote the restrictions of $\varphi$ and $\psi$ to $V_j$ by $\varphi_j$ and $\psi_j$. Corollary \ref{corollary: dimension of irreducible representations of G} gives $\dim V_j\in\{1,3\}$. The relation $\varphi_j\psi_j=\zeta\psi_j\varphi_j$ forces $3\mid\dim V_j$, so $k=2$ and $\dim V_1=\dim V_2=3$.

	Lemma \ref{lemma: structure of semi-commutative pair when N=p} shows that $\varphi^3$ restricts to a scalar on each $V_j$. Since the projective order of $\pi(\varphi)$ is $9$, the matrix $\varphi^3$ is not scalar on $V$; hence its two scalar restrictions are distinct, and $V_1\oplus V_2$ is precisely its eigenspace decomposition. Applying Lemma \ref{lemma: eigenspace decomposition induce split} to the order-three element $\varphi^3$ shows that $F$ splits. For $i=1,2$, let $F_i=F|_{V_i}$ and define $G_i\subseteq\PGL(V_i)$ as in Lemma \ref{lemma: liftability and decompositions}. Lemma \ref{lemma: F-liftability and decompositions} gives $G_i\subseteq\Lin(F_i)$. Lemma \ref{lemma: F-liftability for split F when p=3} then implies that one of $\Lin(F_1)$ and $\Lin(F_2)$ is $F_i$-liftable. The corresponding group $G_i$ is liftable, so Lemma \ref{lemma: liftability and decompositions} makes $G$ liftable. Every $2$-subgroup of $\Lin(F)$ is liftable by Corollary \ref{corollary: liftability of p groups when gcd(d,N,p)=1}(2), and the Sylow criterion now makes $\Lin(F)$ liftable, a contradiction.
\end{proof}

We now prove the $F$-liftability statement.

\begin{proof}[Proof of Theorem \ref{main theorem: cubic fourfold criteria}(2)]
	Only the sufficiency requires proof. By Theorem \ref{main theorem: cubic fourfold criteria}(1), $\Lin(F)$ is liftable. It suffices to prove that a Sylow $3$-subgroup $G$ of $\Lin(F)$ is $F$-liftable. If $G$ is trivial, there is nothing to prove. Choose an element $g$ of order $3$ in the center of $G$. By hypothesis, $g$ has an $F$-lifting $\widetilde g$. Consider its eigenspace decomposition:
	\begin{equation*}
		V = V_{1} \oplus \cdots \oplus V_{k}.
	\end{equation*} 
	Since $\Lin(F)$ is liftable, every subgroup $\langle g,h\rangle$ with $h\in G$ is liftable. Lemma \ref{lemma: the space decomposition} therefore shows that all $ V_{i} $ are $\pi^{-1}(G)$-invariant. Define $ G_{i} \subseteq \PGL(V_{i}) $ as in Lemma \ref{lemma: liftability and decompositions} and let $ F_{i} $ be the restriction of $ F $ to $ V_{i} $.

	If some $\dim V_i$ is not divisible by $3$, Lemma \ref{lemma: decomposition of F} shows that $F_i$ is nonsingular, and Lemma \ref{lemma: F-liftability and decompositions} gives $G_i\subseteq\Lin(F_i)$. Corollary \ref{corollary: liftability of p groups when gcd(d,N,p)=1}(2) then shows that $G_i$ is $F_i$-liftable. Lemma \ref{lemma: F-liftability and decompositions} then makes $G$ $F$-liftable.
	
	Otherwise, $k=2$ and $\dim V_1=\dim V_2=3$ because $g$ is nontrivial. Lemma \ref{lemma: eigenspace decomposition induce split} shows that $F$ splits, and Lemma \ref{lemma: F-liftability and decompositions} gives $G_i\subseteq\Lin(F_i)$ for $i=1,2$. The hypothesis implies that all abelian $3$-subgroups of order at most $9$ are liftable, so Lemma \ref{lemma: F-liftability for split F when p=3} shows that one of $\Lin(F_1)$ and $\Lin(F_2)$ is $F_i$-liftable. The corresponding $G_i$ is $F_i$-liftable, and Lemma \ref{lemma: F-liftability and decompositions} now makes $G$ $F$-liftable.
\end{proof}

\section{Sharpness Examples}
\label{section: example}
The bounded-order criteria above are sharp in general. For every prime $p\geq5$, the examples in this section show that the bound $p^3$ in Theorems \ref{main theorem: subgroup criterion of liftability when N=2p} and \ref{main theorem: subgroup criterion of F-liftability when N=2p} cannot be lowered.

We begin with a Hessian-invariance observation used to determine $\Lin(F)$ for sparse forms.

\begin{lem}
	\label{lemma: second derivative}
		Fix a basis $\{e_{1},e_{2},\cdots,e_{N}\}$ of $V$ and let $\{x_{1},x_{2},\cdots,x_{N}\}$ be its dual basis. Let $F\in S^{d}(V^{*})$ be a nonsingular homogeneous form of degree $d\geq2$. Assume $\varphi \in \GL(V)$ and $\lambda \in \KK^\times$ satisfy $F \circ \varphi = \lambda F$. Then, for every $1\leq k,l \leq N$, the polynomial $\frac{\partial^{2}F}{\partial x_{k}\partial x_{l}}\circ \varphi$ lies in the $\KK$-linear span of the polynomials $\frac{\partial^{2}F}{\partial x_{i}\partial x_{j}}$, where $1 \leq i,j \leq N$.   
\end{lem}
\begin{proof}
		Denote by $\nabla^{2}F$ the matrix $(\frac{\partial^{2}F}{\partial x_{i}\partial x_{j}})$. Let $A = (a_{ij})$ be the matrix representing $\varphi$, so that, for each $1 \leq k \leq N$,  
	\begin{equation*}
		\varphi (e_{k}) = \sum\limits_{i=1}^{N}a_{ki}e_{i}.
	\end{equation*}
	
	Therefore, for $1 \leq k \leq N$, 
	\begin{equation*}
		\varphi^{*}(x_{k}) = x_{k}\circ\varphi = \sum\limits_{i=1}^{N} a_{ik}x_{i}.
	\end{equation*}
	
	The chain rule gives, for $1 \leq k,l \leq N$,
	\begin{equation*}
		\lambda \frac{\partial^{2}F}{\partial x_{k}\partial x_{l}} = \frac{\partial^{2}(F\circ \varphi)}{\partial x_{k}\partial x_{l}} = \sum\limits_{i,j=1}^{N} a_{ki}a_{lj}(\frac{\partial^{2}F}{\partial x_{i}\partial x_{j}} \circ \varphi).
	\end{equation*}
	
	Equivalently,
	\begin{equation*}
		\lambda \nabla^{2}F = A\cdot(\nabla^{2}F\circ \varphi)\cdot A^{T},  
	\end{equation*}
	where $A^{T}$ is the transpose of $A$.
	
	Consequently,
	\begin{equation*}
		\nabla^{2}F \circ \varphi = \lambda A^{-1}\cdot\nabla^{2}F \cdot (A^{-1})^{T},
	\end{equation*}
	Thus, for every $1\leq k,l \leq N$, the polynomial $\frac{\partial^{2}F}{\partial x_{k}\partial x_{l}}\circ \varphi$ lies in the $\KK$-linear span of the polynomials $\frac{\partial^{2}F}{\partial x_{i}\partial x_{j}}$, where $1 \leq i,j \leq N$.
\end{proof}

Now suppose $N=2p$, where $p \geq 5$ is a prime. Fix a primitive $p^{2}$-th root of unity $\xi\in\KK$ and set $\zeta=\xi^p$. Fix a basis of $V$ and let $\{y_{1},\cdots,y_{p},z_{1},\cdots,z_{p}\}$ be its dual basis. Whenever a cyclic expression occurs, its subscripts are read modulo $p$.

Consider the homogeneous form
\begin{equation*}
	F_{1} = y_{1}y_{2}^{p-1}  + \cdots + y_{p}y_{1}^{p-1}
	+ z_{1}^{p} + \cdots + z_{p}^{p} + z_{1}\cdots z_{p}.
\end{equation*}
It is nonsingular. Indeed, if the partial derivatives of the $y$-part vanish and one $y_i$ is zero, then all $y_i$ vanish; if they are all nonzero, multiplying the derivative equations gives $1=-(p-1)^p$, a contradiction. The same argument for the $z$-part gives either $z_1=\cdots=z_p=0$ or $1=-p^{-p}$. Thus each summand has no nonzero critical point, and neither does $F_1$.

Consider the following two elements $\psi_{1},\psi_{2}$ of $\GL(V)$: 
\begin{equation*}
	\begin{aligned}
		&\psi_{1}^{*}(y_{i}) = y_{i+1},\ \forall\ 1\leq i \leq p;\ \psi_{1}^{*}(z_{j}) = z_{j+1},\ \forall\ 1\leq j \leq p, \\
		&\psi_{2}^{*}(y_{i}) = \xi^{ip+1}y_{i},\ \forall\ 1\leq i \leq p;\ \psi_{2}^{*}(z_{j}) = \zeta^{j}z_{j},\ \forall\ 1\leq j \leq p. \\
	\end{aligned} 
\end{equation*}

A direct check on the monomials of $F_1$ gives
$\psi_1^*(F_1)=\psi_2^*(F_1)=F_1$. Moreover, for all $i$ and $j$,
\begin{equation*}
	\psi_2^*\psi_1^*=\zeta\psi_1^*\psi_2^*,\qquad
	(\psi_2^*)^p(y_i)=\zeta y_i,\qquad
	(\psi_2^*)^p(z_j)=z_j,
\end{equation*}
and $(\psi_1^*)^p=(\psi_2^*)^{p^2}=\I_{V^*}$. Since pullback reverses
composition, the first identity gives $[\psi_1,\psi_2]=\zeta\I_V$. Hence
$\psi_1$ has order $p$ and $\psi_2$ has order $p^2$. The displayed action of
$(\psi_2^*)^p$ is not scalar, so $\pi(\psi_2)$ also has order $p^2$.
Nontrivial powers of $\psi_1^*$ permute the variables, whereas powers of
$\psi_2^*$ are diagonal, so the cyclic subgroups generated by their
projective images intersect trivially. Consequently,
\begin{equation*}
	\langle\pi(\psi_1),\pi(\psi_2)\rangle\simeq\ZZ/p\ZZ\times\ZZ/p^2\ZZ.
\end{equation*}
This is an abelian subgroup of $\Lin(F_1)$ of order $p^3$. Multiplying either matrix by a scalar does not change their nontrivial commutator, so this subgroup is not liftable. In particular, $\Lin(F_1)$ is not liftable.

We now prove that all $p$-subgroups of $\Lin(F_{1})$ of order at most $p^{2}$ are $F_{1}$-liftable.

\begin{lem}
	\label{lemma: elements in Lin(F1)}
	Assume that $\varphi\in\GL(V)$ and $\lambda\in\KK^\times$ satisfy $F_1\circ\varphi=\lambda F_1$ and $\varphi^p=\I_V$. Then $\lambda=1$ and, with indices read modulo $p$, $\varphi$ has the form
	\begin{equation*}
		\begin{aligned}
			&\varphi^{*}(y_{i}) = \mu(i)y_{k+i},\ \forall\ 1\leq i \leq p; \\
			&\varphi^{*}(z_{j}) = \delta(j)z_{\tau(j)},\ \forall\ 1\leq j \leq p, \\
		\end{aligned} 
	\end{equation*}
	where $0\leq k\leq p-1$, $\tau\in S_p$ satisfies $\tau^p=\id$, and $\mu(i),\delta(j) \in \KK^\times$ satisfy
	\begin{equation*}
		\mu(1)\mu(2)^{p-1} = \cdots = \mu(p)\mu(1)^{p-1} = \delta(1)^{p} = \cdots = \delta(p)^{p} = \delta(1)\cdots\delta(p)=1.
	\end{equation*}
\end{lem}

\begin{proof}
	Direct differentiation gives, for $1 \leq i,j \leq p$,
	\begin{equation*}
		\begin{aligned}
			&\frac{\partial^{2}F_{1}}{\partial y_{i}\partial y_{j}} = \begin{cases}
				(p-1)y_{i}^{p-2},\ j=i-1\\
				(p-1)(p-2)y_{i-1}y_{i}^{p-3},\ j=i\\
				(p-1)y_{i+1}^{p-2},\ j=i+1\\
					0,\ \text{otherwise},\\
			\end{cases};\\  &\frac{\partial^{2}F_{1}}{\partial z_{i}\partial z_{j}} = \begin{cases}
				p(p-1)z_{i}^{p-2},\ j=i\\
					\prod\limits_{t \neq i,j}z_{t},\ \text{otherwise}.\\
			\end{cases};\\
			& \frac{\partial^{2}F_{1}}{\partial y_{i}\partial z_{j}} = 0.
		\end{aligned}
	\end{equation*}
	
	Suppose that 
	\begin{equation*}
		\begin{aligned}
			& \varphi^{*}(y_{k}) = \sum\limits_{i=1}^{p}a_{ik}y_{i} + \sum\limits_{j=1}^{p}b_{jk}z_{j},\ \forall\ 1 \leq k \leq p; \\
			& \varphi^{*}(z_{l}) = \sum\limits_{i=1}^{p}c_{il}y_{i} + \sum\limits_{j=1}^{p}d_{jl}z_{j},\ \forall\ 1 \leq l \leq p. \\
		\end{aligned}
	\end{equation*}
	
		By Lemma \ref{lemma: second derivative}, for every $1\leq k \leq p$, the polynomial
	\begin{equation*}
		\frac{\partial^{2}F_{1}}{\partial y_{k}\partial y_{k-1}} \circ \varphi
		=(p-1)\left(\sum_{i=1}^{p}a_{ik}y_{i} + \sum_{j=1}^{p}b_{jk}z_{j}\right)^{p-2}
	\end{equation*}
	is a $\KK$-linear combination of $y_{i}^{p-2}$, $y_{i-1}y_{i}^{p-3}$, $z_{j}^{p-2}$, and $\prod_{t \neq i,j}z_{t}$ for $i\neq j$. If the linear form in parentheses contained two variables with nonzero coefficients, its $(p-2)$-nd power would contain at least two distinct mixed powers of those variables, with nonzero coefficients because $\operatorname{char}(\KK)=0$ and $p-2\geq3$. At least one of these mixed powers is absent from the displayed Hessian span, a contradiction. Thus at most one term in $\{a_{1k},\cdots,a_{pk},b_{1k},\cdots,b_{pk}\}$ is nonzero. Applying the same argument to $\frac{\partial^2F_1}{\partial z_l^2}\circ\varphi$ gives the analogous conclusion for $\{c_{1l},\cdots,c_{pl},d_{1l},\cdots,d_{pl}\}$. Since $\varphi$ is invertible, its pullback is therefore a monomial transformation: it induces a permutation $\sigma$ of $\{y_{1},\cdots,y_{p},z_{1},\cdots,z_{p}\}$ and has the form
	\begin{equation*}
		\begin{aligned}
			&\varphi^{*}(y_{i}) = \mu(i)\sigma(y_{i}),\ \forall\ 1\leq i \leq p; \\
			&\varphi^{*}(z_{j}) = \delta(j)\sigma(z_{j}),\ \forall\ 1\leq j \leq p, \\
		\end{aligned}
	\end{equation*}
	where $\mu(i),\delta(j) \in \KK^\times$. 
	
	Distinct monomials remain distinct under a monomial transformation, so no cancellation occurs among their images. The monomial $z_1\cdots z_p$ is the unique square-free monomial of degree $p$ in the support of $F_1$. Hence $\sigma(z_1)\cdots\sigma(z_p)=z_1\cdots z_p$, and $\sigma$ preserves the two sets $\{y_{1},\cdots,y_{p}\}$ and $\{z_{1},\cdots,z_{p}\}$. 
	
	Comparing the monomial support of the $y$-part of $F_1$ shows that the restriction of $\sigma$ to $\{y_1,\ldots,y_p\}$ is an automorphism of the directed $p$-cycle, hence a cyclic shift. On the $z$-variables, the equality $\sigma^p=\id$ implies that the restriction is either the identity or a $p$-cycle. Thus there exist $0\leq k\leq p-1$ and a permutation $\tau\in S_p$ with $\tau^p=\id$ such that
	\begin{equation*}
		\begin{aligned}
			&\sigma(y_{i}) = y_{i+k},\ \forall\ 1\leq i \leq p; \\
			&\sigma(z_{j}) = z_{\tau(j)},\ \forall\ 1\leq j \leq p. \\
		\end{aligned}
	\end{equation*}
	
	Since $F_{1}\circ\varphi = \lambda F_{1}$, we have:
	\begin{equation*}
		\mu(1)\mu(2)^{p-1} = \cdots = \mu(p)\mu(1)^{p-1} = \delta(1)^{p} = \cdots = \delta(p)^{p} = \delta(1)\cdots\delta(p)=\lambda.
	\end{equation*}

	It remains to prove that $\lambda=1$. If $\tau=\id$, then $z_1=(\varphi^*)^p(z_1)=\delta(1)^p z_1$, so $\lambda=\delta(1)^p=1$. Otherwise $\tau$ is a $p$-cycle, and $z_1=(\varphi^*)^p(z_1)=\delta(1)\cdots\delta(p)z_1$, so again $\lambda=1$.
\end{proof}

\begin{cor}
	\label{corollary: example for main theorem}
	The group $\Lin(F_1)$ is not liftable, whereas all its $p$-subgroups of order at most $p^2$ are $F_1$-liftable.
\end{cor}

\begin{proof}
		It remains to verify the assertion about the small subgroups. First let $g\in\Lin(F_1)$ have order $p^r$. If $r\geq2$, Theorem \ref{theorem: order of elements in Lin(F)} shows that $g$ is $F_1$-liftable, since every non-$F_1$-liftable $p$-element has order $p$. If $r=1$, choose an order-$p$ lifting $\varphi$ of $g$; Lemma \ref{lemma: elements in Lin(F1)} shows that $F_1\circ\varphi=F_1$. Thus every cyclic $p$-subgroup of $\Lin(F_1)$ is $F_1$-liftable.
	
	Now let $G$ be a noncyclic $p$-subgroup of order at most $p^2$. Then $G\simeq(\ZZ/p\ZZ)^2$. Choose generators $g_1,g_2$ and $F_1$-liftings $\widetilde g_1,\widetilde g_2$. By Lemma \ref{lemma: elements in Lin(F1)}, we may write
	\begin{equation*}
		\begin{aligned}
			&\widetilde{g_{m}}^{*}(y_{i}) = \mu_{m}(i)y_{i+k_{m}},\ \forall\ 1\leq i \leq p; \\
			&\widetilde{g_{m}}^{*}(z_{j}) = \delta_{m}(j)z_{\tau_m(j)},\ \forall\ 1\leq j \leq p, \\
		\end{aligned} 
	\end{equation*}
	where $0\leq k_{m}\leq p-1$, $\tau_m^p=\id$, and $\mu_{m}(i),\delta_{m}(j) \in \KK^\times$ satisfy
	\begin{equation*}
		\mu_{m}(1)\mu_{m}(2)^{p-1} = \cdots = \mu_{m}(p)\mu_{m}(1)^{p-1} = \delta_{m}(1)^{p} = \cdots = \delta_{m}(p)^{p} = \delta_{m}(1)\cdots\delta_{m}(p)=1.
	\end{equation*}
	
	Scalar rescaling does not change the permutation part of a lift, so the shift $k(g)\in\ZZ/p\ZZ$ depends only on $g\in G$. Composition adds shifts, and hence $k\colon G\to\ZZ/p\ZZ$ is a homomorphism. Its kernel contains a nontrivial element, so we may choose the generators such that $k_1=0$. Since $\widetilde g_1^p=\I_V$, we have
	\begin{equation*}
		\mu_{1}(1)^{p}=\mu_{1}(2)^{p}=\cdots=\mu_{1}(p)^{p}=1.
	\end{equation*}
	
	The relations $\mu_1(i)\mu_1(i+1)^{p-1}=1$ now give $\mu_1(i)=\mu_1(i+1)$ for every $i$. Thus $\mu_{1}(1) = \mu_{1}(2) = \cdots = \mu_{1}(p)$, and
	\begin{equation*}
		[\widetilde{g_{2}},\widetilde{g_{1}}]^{*}(y_{i})= y_{i} \circ [\widetilde{g_{2}},\widetilde{g_{1}}]= y_{i},\ \forall\ 1\leq i \leq p.
	\end{equation*}
	
	Since $\pi([\widetilde g_2,\widetilde g_1])=[g_2,g_1]=\id$, the commutator is scalar; the preceding equality on the $y$-variables forces this scalar to be $1$. Hence $\widetilde g_1$ and $\widetilde g_2$ commute. Their projective images are independent generators of order $p$, so the group they generate maps isomorphically onto $G$ and is an $F_1$-lifting.
\end{proof}

\begin{rmk}
	For every prime $p\geq5$, the form $F_1$ gives a group $\Lin(F_1)$ that is neither liftable nor $F_1$-liftable, although all its $p$-subgroups of order at most $p^2$ are $F_1$-liftable. Consequently, the bound $p^3$ in Theorems \ref{main theorem: subgroup criterion of liftability when N=2p} and \ref{main theorem: subgroup criterion of F-liftability when N=2p} is optimal.
\end{rmk}

\bibliography{reference}
\end{document}